\documentclass[letterpaper,11pt]{amsart}
\usepackage{amsmath,amsthm,amssymb,mathtools,amsfonts,mathrsfs}
\usepackage{xspace,xcolor}
\usepackage[breaklinks,colorlinks,citecolor=teal,linkcolor=teal,urlcolor=teal,pagebackref,hyperindex]{hyperref}
\usepackage[alphabetic]{amsrefs}
\usepackage{tikz-cd}
\usepackage[all]{xy}
\usepackage{bbm}

\newtheorem{theorem}{Theorem}[section]

\newtheorem{lemma}[theorem]{Lemma}

\theoremstyle{definition} 
\newtheorem{defn}[theorem]{Definition}

\newtheorem{example}[theorem]{Example}

\newtheorem{remark}[theorem]{Remark}

\newcommand{\qu}{/\kern-.7ex/}
\newcommand{\lqu}{\backslash \kern-.7ex \backslash}
\newcommand{\on}{\operatorname}

\newcommand{\HH}{\mathfrak H}

\title[GW of $X_{D,r}$ with mid-ages and the loop axiom]{Gromov--Witten invariants of root stacks with mid-ages and the loop axiom}

\author{Fenglong You}
\address{Department of Mathematics\\ Imperial College\\ London SW7 2AZ\\ United
Kingdom}
\email{f.you@imperial.ac.uk}

\thanks{}


\begin{document}
\date{\today}

\begin{abstract} 
We study orbifold Gromov--Witten invariants of the $r$-th root stack $X_{D,r}$ with a pair of mid-ages when $r$ is sufficiently large. We prove that genus $g$ invariants with a pair of mid-ages $k_a/r$ and $1-k_a/r$ are polynomials in $k_a$ and the $k_a^i$-coefficients are polynomials in $r$ with degree bounded by $2g$. Moreover, genus zero invariants with a pair of mid-ages are independent of the choice of mid-ages. As an application, we obtain an identity for relative Gromov--Witten theory which can be viewed as a modified version of the usual loop axiom.
\end{abstract}

\maketitle 

\tableofcontents

\section{Introduction}

\subsection{Overview}
Given a smooth projective variety $X$ with a smooth effective divisor $D$, the root construction is essentially the only way stack structures can arise in codimension one. The $r$-th root stack is denoted by $X_{D,r}$. The moduli space of stable maps to the root stack $X_{D,r}$ provides an alternative compactification of an open substack of the moduli space of relative stable maps to the pair $(X,D)$ \cite{Cadman}. Therefore, orbifold Gromov--Witten invariants of $X_{D,r}$ are closely related to relative Gromov--Witten invariants of $(X,D)$. Results in \cite{ACW}, \cite{TY18} and \cite{TY20} stated that, for sufficiently large $r$, genus $g$ orbifold Gromov--Witten invariants with small ages are polynomials in $r$ with degree bounded by $\max\{2g-1,0\}$ and the constant terms are the corresponding relative Gromov--Witten invariants with positive contact orders. On the other hand, by \cite{FWY}, \cite{FWY19} and \cite{TY20}, genus $g$ orbifold Gromov--Witten invariants with large ages, after multiplying by suitable powers of $r$, are polynomials in $r$ with degree bounded by $\max \{2g-1,0\}$ and the constant terms are the corresponding relative Gromov--Witten invariants with negative contact orders defined in \cite{FWY} and \cite{FWY19}. In particular, by \cite{ACW}, \cite{TY18} and \cite{FWY}, genus zero orbifold invariants,  after multiplying by suitable powers of $r$, are equal to genus zero relative invariants when $r$ is sufficiently large. The type of orbifold invariants of $X_{D,r}$ that has not been explored is the type of orbifold invariants which include orbifold markings with neither small nor large ages. These invariants will be called orbifold Gromov--Witten invariants of $X_{D,r}$ with mid-ages. The purpose of this paper is to understand these invariants. 

One motivation for studying orbifold Gromov--Witten invariants with mid-ages is to find a modified version of the loop axiom for relative Gromov--Witten theory. It was proved in \cite{FWY19} that relative Gromov--Witten theory is a partial cohomological field theory (partial CohFT) in the sense of \cite{LRZ}, which is a CohFT without the loop axiom. A counterexample for the loop axiom for relative Gromov--Witten theory is given in \cite{FWY19}*{Example 3.17} (See also Example \ref{example} for how it can be modified). The loop axiom is related to the quantization and Virasoro constraints. Therefore, one possible future work after our result is to study the Virasoro constraints for relative Gromov--Witten theory.

\subsection{Set-up}

Given a smooth projective variety $X$ with a smooth effective divisor $D$ and $\beta \in H_2(X)$. Let $\vec k=(k_1,\ldots,k_m)$ be a vector of $m$ integers which satisfy 
\[
\sum_{i=1}^m k_i=\int_\beta[D].
\]
The number of positive elements, zero elements and negative elements in $\vec k$ are denoted by $m_+, m_0$ and $m_-$ respectively. So $m=m_++m_0+m_-$.

Evaluation maps for orbifold Gromov--Witten theory of $\mathcal X$ land on the inertia stack $I\mathcal X$ of the target orbifold $\mathcal X$. The coarse moduli space $\underline{I}X_{D,r}$ of the inertia stack of the root stack $X_{D,r}$ can be decomposed into disjoint union of $r$ components
\[
\underline{I}X_{D,r}=X\sqcup \coprod_{i=1}^{r-1} D,
\] 
labeled by ages: $0, 1/r, 2/r, \ldots, (r-1)/r$.

We assume that $r>|k_i|$ for all $i\in\{1,\ldots, m\}$. We consider the moduli space $\overline{M}_{g,\vec k,\beta}(X_{D,r})$ of $m$-pointed, genus $g$, degree $\beta\in H_2(X,\mathbb Q)$, orbifold stable maps to $X_{D,r}$ where the $i$-th marking is an interior marking if $k_i=0$; the $i$-th orbifold marking maps to the twisted sector of the inertia stack of $X_{D,r}$ with age $k_i/r$ if $k_i>0$; the $i$-th orbifold marking maps to the twisted sector of the inertia stack of $X_{D,r}$ with age $(r+k_i)/r$ if $k_i<0$.

Let
\begin{itemize}
    \item $\gamma_{i}\in H^*(\underline{I}X_{D_,r},\mathbb Q)$ for $1\leq i \leq m$;
    \item $a_i\in \mathbb Z_{\geq 0}$, for $1\leq i \leq m$.
\end{itemize}
Then orbifold Gromov--Witten invariants of $X_{D,r}$ are defined as

\begin{align}\label{orb-inv}
\left\langle \prod_{i=1}^m\tau_{a_{i}}(\gamma_{i})\right\rangle^{X_{D,r}}_{g,\vec k,\beta}:=
\int_{[\overline{M}_{g,\vec k,\beta}(X_{D,r})]^{\on{vir}}}\bar{\psi}_1^{a_1}\on{ev}^*_{1}(\gamma_{1})\cdots \bar{\psi}_m^{a_m}\on{ev}^*_{m}(\gamma_{m}),
\end{align}
where the descendant class $\bar{\psi}_i$ is the class pullback from the corresponding descendant class on the moduli space $\overline{M}_{g,m,\beta}(X)$ of stable maps to $X$.

Let 
\begin{align}\label{rel-inv}
\left\langle \prod_{i=1}^m\tau_{a_{i}}(\gamma_{i})\right\rangle^{(X,D)}_{g,\vec k,\beta}
\end{align}
be the corresponding relative Gromov--Witten invariants of $(X,D)$ with contact orders $k_i$ at the $i$-th marking, for $1\leq i\leq m$. When $m_-=0$, invariants (\ref{rel-inv}) are simply the relative Gromov--Witten invariants without negative contact orders defined in \cite{LR}, \cite{IP}, \cite{Li1} and \cite{Li2}. When $m_->0$, invariants (\ref{rel-inv}) are the relative Gromov--Witten invariants with negative contact orders defined in \cite{FWY} and \cite{FWY19}. 

Given a partition $\vec k$, we take a sufficiently large $r$. We consider orbifold invariants of root stack $X_{D,r}$ with ages 
\[
k_1/r,\ldots, k_m/r, k_a/r,k_b/r,
\]
where $k_a,k_b\in \{1,2,\ldots, r-1\}$ and $k_a+k_b=r$. In other words, we consider orbifold invariants with two extra orbifold markings. 
Without loss of generality, we will always have $k_a\leq k_b$ in Section \ref{sec:genus-0} and Section \ref{sec:higher-genus}. 

Orbifold invariants with a pair of extra markings are denoted by
\begin{align}\label{inv-mid-age}
\left\langle \tau_0(\gamma_a)\tau_0(\gamma_b)\prod_{i=1}^m\tau_{a_{i}}(\gamma_{i})\right\rangle^{X_{D,r}}_{g,\vec k,k_a,k_b,\beta}.
\end{align}

\begin{remark}
Note that, we can also choose more than two numbers in $\{1,2\ldots, r-1\}$. For example, we can take $k_a,k_b,k_c\in \{1,2,\ldots, r-1\}$ and $k_a+k_b+k_c=r$. Most of the results of this paper remain the same. In our paper, we only consider invariants with a pair of mid-ages because these are the invariants that appear in the loop axiom of orbifold Gromov--Witten theory. There are also orbifold invariants with more than one pair of mid-ages, e.g.  $k_a,k_b,k_c,k_d\in \{1,2,\ldots, r-1\}$ and $k_a+k_b=r$, $k_c+k_d=r$. We do not consider such invariants, because they will not appear in the loop axiom if we begin with orbifold invariants with only small ages and large ages. Similar results for these invariants can be obtained following the proof of Theorem \ref{thm-genus-zero-mid-age} and Theorem \ref{thm-higher-genus-mid-age}.
\end{remark}

\subsection{Main results}
We consider the forgetful map
\[
\tau: \overline{M}_{g,\vec k, k_a,k_b,\beta}(X_{D,r}) \rightarrow \overline{M}_{g,m+2,\beta}(X)\times_{X^{m_++m_-+2}}D^{m_++m_-+2},
\]
which forgets the orbifold structures. For genus zero invariants, we have

\begin{theorem}\label{thm-genus-zero-mid-age}
The genus zero cycle class
\[
r^{m_-+1}\tau_*\left(\left[\overline{M}_{0,\vec k, k_a,k_b,\beta}(X_{D,r})\right]^{\on{vir}}\right)\in A_*\left(\overline{M}_{0,m+2,\beta}(X)\times_{X^{m_++m_-+2}}D^{m_++m_-+2}\right)
\] 
of the root stack $X_{D,r}$ is constant in $r$ for sufficiently large $r$. Furthermore, there exists a positive integer $d_0$ such that this cycle class does not depend on the pair $(k_a,k_b)$ as long as $d_0<k_a\leq k_b$.
\end{theorem}

For higher genus invariants, we have

\begin{theorem}\label{thm-higher-genus-mid-age}
The genus $g$ cycle class
\[
\tau_*\left(\left[\overline{M}_{g,\vec k, k_a,k_b,\beta}(X_{D,r})\right]^{\on{vir}}\right)\in A_*\left(\overline{M}_{g,m+2,\beta}(X)\times_{X^{m_++m_-+2}}D^{m_++m_-+2}\right)
\]
is a polynomial in $k_a$. Furthermore, the cycle class
\[
r^{m_-+1}\tau_*\left(\left[\overline{M}_{g,\vec k, k_a,k_b,\beta}(X_{D,r})\right]^{\on{vir}}_{k_a^i}\right)\in A_*\left(\overline{M}_{g,m+2,\beta}(X)\times_{X^{m_++m_-+2}}D^{m_++m_-+2}\right)
\]  
of the root stack $X_{D,r}$ is a polynomial in $r$ with degree bounded by $2g$ when $r$ is sufficiently large, where $[\cdots]_{k_a^i}$ is the $k_a^i$-coefficient of the polynomial in $k_a$.
\end{theorem}

\begin{remark}
We only talk about a bound for the degree. There maybe better bounds. Since we are mostly interested in taking the constant term of the polynomial, we do not plan to study the precise upper bound of the degree.  At the level of invariants, since insertions provide extra constraints, there maybe better bounds. Such examples include stationary invariants of target curves considered in \cite{TY18}*{Theorem 1.9}. See also Example \ref{example-stationary}.
\end{remark}

\begin{remark}
Following \cite{FWY19}*{Theorem 3.13} and \cite{TY20}*{Theorem 1.1}, given $k_a$ (hence $k_b$ is also given), for sufficiently large $r$, the invariant (\ref{inv-mid-age}) is a polynomial in $r$ with degree bounded by $2g-1$ and the constant term is the corresponding relative invariant with negative contact orders. However, if we also increase $k_a$ as $r$ increases, we only know the degree of the polynomial is bounded by $2g$. Therefore, Theorem \ref{thm-higher-genus-mid-age} and the loop axiom for orbifold Gromov--Witten theory provide another explanation for the degree bound of genus $(g+1)$ orbifold invariants. In particular, Theorem \ref{thm-genus-zero-mid-age} shows that genus zero invariants are constant in $r$. Together with the loop axiom for genus one invariants, we have another explanation why genus one invariants are linear functions in $r$.  
\end{remark}

\begin{remark}
The difference between genus zero invariants and higher genus invariants can also be seen from the viewpoint of double ramification cycles. Genus zero double ramification cycle is simply the identity class, hence does not depend on $\vec k$, $k_a$ or $k_b$. While higher genus double ramification cycles are polynomials in $\vec k$, $k_a$ and $k_b$ by \cite{JPPZ}*{Appendix A.4}. Since invariants in Theorem \ref{thm-higher-genus-mid-age} are closely related to double ramification cycles, we do not expect they are independent of the pair $(k_a,k_b)$. 
\end{remark}

The relationship between relative and orbifold invariants can be used to state a modified version of the loop axiom for relative Gromov--Witten theory. In particular, the genus zero result in Theorem \ref{thm-genus-zero-mid-age} implies the loop axiom for genus one relative Gromov--Witten theory.

Let $\HH=\bigoplus\limits_{i\in\mathbb Z}\HH_i,$ be the ring of insertions of relative Gromov--Witten theory. Let $\{[e_i]\}$ be a basis of $\HH$, $\eta_{jk}=\langle [e_j],[e_k]\rangle$ and $(\eta^{jk})=(\eta_{jk})^{-1}$. Consider the morphism
  \[
    \rho_l: \overline{M}_{g,m+2}\rightarrow \overline{M}_{g+1,m}
    \]
    obtained by identifying the last two markings of the $(m+2)$-pointed, genus $g$ curves. Then the following modified loop axiom holds.

\begin{theorem}\label{thm-genus-one-loop-axiom}
Genus one relative Gromov--Witten theory satisfies the following identity which can be viewed as a modified version of the usual loop axiom:
\begin{align*}
&\rho_l^*\Omega^{(X,D)}_{1,m}([\gamma_1],\ldots,[\gamma_m])\\
=&\sum_{j,k: [e_j]\in \HH_{k_a}, k_a\leq d_0 \text{ or } k_a\geq r-d_0}\eta^{jk}\Omega^{(X,D)}_{0,m+2}([\gamma_1],\ldots,[\gamma_m],[e_j],[e_k])\\
&-(2d_0+1)\left[r^{m_-}\rho_l^*\Omega^{X_{D,r}}_{1,m}(\gamma_1,\ldots,\gamma_m)\right]_{r^1}.
\end{align*}
for all sufficiently large integers $d_0$, where $\Omega^{(X,D)}_{g,m}$ is the relative Gromov--Witten class defined in Definition \ref{defn-relative-GW-class}.
\end{theorem}

In \cite{TY18}, it was showed that orbifold Gromov--Witten invariants of $X_{D,r}$ are polynomials in $r$ for sufficiently large $r$ and the constant terms are the corresponding relative Gromov--Witten invariants of $(X,D)$. In \cite{TY20}, it was proved that the degrees of the polynomials are bounded by $(2g-1)$ for genus $g$ Gromov--Witten invariants and $r^i$-coefficients of the polynomials are lower genus relative Gromov--Witten invariants for $i>0$. Theorem \ref{thm-genus-one-loop-axiom} provides another meaning for the $r$-coefficients of the polynomials for genus one orbifold Gromov--Witten invariants.

The result for higher genus orbifold Gromov--Witten theory in Theorem \ref{thm-higher-genus-mid-age} implies the following identity which can be viewed as a modification of the usual loop axiom for higher genus relative Gromov--Witten theory.
\begin{theorem}\label{thm-higher-genus-loop-axiom}
Higher genus relative Gromov--Witten theory of $(X,D)$ satisfies the following identity,  for sufficiently large integer $d_0$,
\begin{align*}
&\rho_l^*\Omega^{(X,D)}_{g+1,m}([\gamma_1],\ldots,[\gamma_m])\\
=&\sum_{j,k: [e_j]\in \HH_{k_a}, k_a\leq d_0 \text{ or }k_a\geq r-d_0}\eta^{jk}\Omega^{(X,D)}_{g,m+2}([\gamma_1],\ldots,[\gamma_m],[e_j],[e_k])+C_{d_0}.
\end{align*}
where $C_{d_0}$, given in Equation (\ref{higher-genus-correction}), is the correction term given by the constant part of the sum of orbifold Gromov--Witten classes with a pair of mid-ages.
\end{theorem}

More explicit form of $C_{d_0}$ is also discussed in Section \ref{sec:higher-genus-loop}. We also computed an example at the level of invariants in Example \ref{example-stationary}.

\subsection{Acknowledgement}
The author would like to thank Honglu Fan and Longting Wu for related collaboration and valuable comments on a draft. The author would also like to thank Hsian-Hua Tseng for related collaboration. F. Y. is supported by a postdoctoral fellowship of NSERC and the Department of Mathematical and Statistical Sciences at the University of Alberta.

\section{Genus zero}\label{sec:genus-0}

In this section, we consider the invariant (\ref{inv-mid-age}) with $g=0$. The goal is to prove Theorem \ref{thm-genus-zero-mid-age}.

Let $L=N_D$ be a normal bundle over $D$ in $X$ and $Y$ be the total space of the $\mathbb P^1$-bundle
\[
\pi: \mathbb P^1(\mathcal O_D\oplus L)\rightarrow D.
\]
The zero and infinity divisors of $Y$ are denoted by $D_0$ and $D_\infty$. 
We apply the $r$-th root construction to $D_0$ to obtain the root stack $Y_{D_0,r}$. The zero and infinity divisors of $Y_{D_0,r}$ are denoted by $\mathcal D_0$ and $ D_\infty$. 

We consider the moduli space $\overline{M}_{0,\vec k,k_a,k_b,\beta}(X_{D,r})$ of $(m+2)$-pointed, genus 0, degree $\beta$, orbifold stable maps to $X_{D,r}$ whose orbifold data is given by the partition $\vec k$, $k_a$ and $k_b$.
Using the degeneration formula, it is sufficient to consider 
\[
[\overline{M}_{0,\vec k,,k_a,k_b,\vec \mu,\beta}(Y_{D_0,r}, D_\infty)]^{\on{vir}},
\]
which is the moduli space of orbifold-relative stable maps to $(Y_{D_0,r}, D_\infty)$ with orbifold data given by $\vec k, k_a,k_b$ and relative data given by $\vec \mu$.
This is the standard argument  in \cite{TY18}, \cite{FWY19} and \cite{FWY} of using the degeneration formula to reduce the computation to relative local models. We do not plan to repeat this argument here.

\subsection{Invariants without large ages}
We first consider the case when $m_-=0$. In other words, the original invariant (\ref{orb-inv}) does not contain orbifold markings with large ages. 

There is a natural $\mathbb C^*$-action on $Y$ which induces a natural $\mathbb C^*$-action on the moduli space $\overline{M}_{0,\vec k,,k_a,k_b,\vec \mu,\beta}(Y_{D_0,r}, D_\infty)$. Therefore, it can be computed by the virtual localization formula studied in \cite{JPPZ18}, \cite{TY18} and \cite{FWY19}. We refer readers to \cite{JPPZ18} for details of the virtual localization formula. A component of the domain curve is called contracted if it lands on the zero section $\mathcal D_0$ or the infinity section $D_\infty$. The $\mathbb C^*$-fixed loci of  $\overline{M}_{0,\vec k,,k_a,k_b,\vec \mu,\beta}(Y_{D_0,r}, D_\infty)$ are in bijective correspondence with decorated graphs. Although this section is for genus zero invariants, we briefly describe a decorated graph for genus $g$ invaraints since decorated graphs are also used for genus $g$ invariants in later sections. A decorated graph is a bipartite connected graph with decorations. A decorated graph $\Gamma$ contains the following data and compatibility conditions (see \cite{TY18}*{Section 3.2.1} for more detail).
\begin{itemize}
\item $V(\Gamma)$ is the set of vertices of $\Gamma$. 
Each vertex $v$ is decorated by the genus $g(v)$ and the degree $d(v)\in H_2(D,\mathbb Z)$. The degree $d(v)$ must be an effective curve class.
Vertices are labelled by either $0$ or $\infty$ corresponding to contracted components over the zero divisor or infinity divisor.The labeling map is denoted by 
\[
i:V(\Gamma)\rightarrow \{0,\infty\}.
\] 
\item $E(\Gamma)$ is the set of edges of $\Gamma$. We write $E(v)$ for the set of edges attached to the vertex $v\in V(\Gamma)$ and write $|E(v)|$ for the number of edges attached to the vertex $v\in V(\Gamma)$.
Each edge $e$ is decorated by the degree $d_e\in \mathbb Z_{ >0}$.
\item The set of legs is in bijective correspondence with the set of markings. We write $S(v)$ to denote the set of markings assigned to the vertex $v$.

\item $\Gamma$ is a connected graph, and $\Gamma$ is bipartite with respect to labeling $i$. Each edge is incident to a vertex labeled by $0$ and a vertex labeled by $\infty$.

\end{itemize}



The following lemma is true for genus $g$ invariants and will be used in later sections.
\begin{lemma}\label{lemma-d-0}
For a sufficiently large integer $r$, there exists a positive integer $d_0$ such that when $d_0<k_a\leq k_b$, the two orbifold markings $p_a$ and $p_b$, with ages $k_a/r$ and $k_b/r$ respectively, lie in the same contracted component over $0$.
\end{lemma}

\begin{proof}
We follow the idea of \cite{JPPZ18}*{Lemma 12}. Define $\beta^\prime\in H_2(D,\mathbb Z)$ to be an effective summand of $\pi_*\beta$ if both $\beta^\prime$ and $\pi_*\beta -\beta^\prime$ are effective cycle classes. Let $b$ be the maximum of $\left|\int_{\beta^\prime}c_1(N_D)\right|$ over all effective summands of $\pi_*\beta$. 

Assume
\[
r>4(\sum_{i=1}^m k_i+b).
\]
For the stable vertex over $0$ containing the orbifold marking $p_b$, suppose $p_a$ is not in this contracted component, then the condition
\[
k_b+\sum_{j\in S(v)}k_i- \int_{\beta(v)}c_1(N_D)=\sum_{e\in E(v)}d_e \mod r
\]
holds. By our choice of $r$, we must have $\sum_{e\in E(v)}d_e<r/4$. Recall that $k_a+k_b=r$ and $k_a\leq k_b$. Hence we have
\begin{align}\label{equ-vertex}
-k_a+\sum_{j\in S(v)}k_i- \int_{\beta(v)}c_1(N_D)=\sum_{e\in E(v)}d_e.
\end{align}
Therefore, there exists an integer $d_0<r/2$ such that when $d_0<k_a\leq r/2$, the equality (\ref{equ-vertex}) can not hold. In this case, orbifold markings $p_a$ and $p_b$ lie in the same contracted component.
\end{proof}

Note that the minimal possible value of $d_0$ does not depend on $r$ as long as $r$ is chosen to be sufficiently large. Invariants with $d_0<k_a\leq k_b$ are called invariants with mid-ages. The above lemma implies the following lemma. 

\begin{lemma}\label{lemma-contracted}
There exists a positive integer $d_0$ such that the decorated graphs of the fixed loci of $\overline{M}_{0,\vec k,,k_a,k_b,\vec \mu,\beta}(Y_{D_0,r}, D_\infty)$ are the same for all $d_0<k_a\leq k_b$.
\end{lemma}

\begin{proof}[Proof of Theorem \ref{thm-genus-zero-mid-age} with $m_-=0$]

For $k_a\leq d_0$, we can find a sufficiently large $r_0$ such that the genus zero orbifold invariants (\ref{inv-mid-age}) are constant in $r$, for any $r>r_0$, and equal to genus zero relative invariants with negative contact order $-k_a$. 

For such a large $r$, we consider the case when $k_a> d_0$. In other words, we consider genus zero invariants with a pair of mid-ages. 

 The $\mathbb C^*$-fixed locus corresponding to the decorated graph $\Gamma$ is denoted by $\overline{M}_{\Gamma}$. There is a natural morphism
\[
\iota:\overline{M}_{\Gamma}\rightarrow \overline{M}_{g,\vec k,k_a,k_b,\vec \mu,\beta}(Y_{D_0,r},D_\infty).
\]
Following \cite{JPPZ18}, \cite{TY18} and \cite{FWY19}, the localization formula is
\begin{align}\label{localization-genus-zero}
[\overline{M}_{0,\vec k,k_a,k_b,\vec \mu,\beta}(Y_{D_0,r}, D_\infty)]^{\on{vir}}= \qquad \sum_{\Gamma}\frac{1}{|\on{Aut}(\Gamma)|\prod_{e\in E(\Gamma)}d_e} \cdot\iota_*\left(\frac{[\overline{M}_{\Gamma}]^{\on{vir}}}{e(\on{Norm}_{\Gamma}^{\on{vir}})}\right).
\end{align}
The inverse of the virtual normal bundle $\frac{1}{e(\on{Norm}_{\Gamma}^{\on{vir}})}$ can be written as the product of the following factors:
\begin{itemize}
\item for each stable vertex $v$ over the zero section that does not contain mid-ages, there is a factor
\begin{align}\label{vertex-contri-genus-zero}
&\prod_{e\in E(v)}\frac{rd_e}{t+\on{ev}_{e}^*c_1(L)-d_e\bar{\psi}_{(e,v)}}\left(\sum_{i=0}^{\infty}(t/r)^{-1+|E(v)|-i}c_i(-R^*\pi_*\mathcal L)\right)\\
\notag =&\left(\frac rt\right)^{|E(v)|}\prod_{e\in E(v)}\frac{d_e}{1+\frac{\on{ev}_{e}^*c_1(L)-d_e\bar{\psi}_{(e,v)}}{t}}\left(\sum_{i=0}^{\infty}(t/r)^{-1+|E(v)|-i}c_i(-R^*\pi_*\mathcal L)\right)\\
\notag =&\frac rt \prod_{e\in E(v)}\frac{d_e}{1+\frac{\on{ev}_{e}^*c_1(L)-d_e\bar{\psi}_{(e,v)}}{t}}\left(\sum_{i=0}^{\infty}(t/r)^{-i}c_i(-R^*\pi_*\mathcal L)\right),
\end{align}
where 
\[
\pi: \mathcal C_{0,\on{val}(v),\beta(v)}(\mathcal D_r)\rightarrow \overline{M}_{0,\on{val}(v),\beta(v)}(\mathcal D_r)
\]
 is the universal curve, 
\[
\mathcal L\rightarrow \mathcal C_{0,\on{val}(v),\beta(v)}(\mathcal D_r)
\] 
is the universal $r$-th root.
\item for the stable vertex $v_0$ over the zero section containing mid-ages, there is a factor
\begin{align}\label{vertex-contri-genus-zero-mid}
&\prod_{e\in E(v_0)}\frac{rd_e}{t+\on{ev}_{e}^*c_1(L)-d_e\bar{\psi}_{(e,v_0)}}\left(\sum_{i=0}^{\infty}(t/r)^{|E(v_0)|-i}c_i(-R^*\pi_*\mathcal L)\right)\\
\notag =&\left(\frac rt\right)^{|E(v_0)|}\prod_{e\in E(v_0)}\frac{d_e}{1+\frac{\on{ev}_{e}^*c_1(L)-d_e\bar{\psi}_{(e,v_0)}}{t}}\left(\sum_{i=0}^{\infty}(t/r)^{|E(v_0)|-i}c_i(-R^*\pi_*\mathcal L)\right)\\
\notag =& \prod_{e\in E(v_0)}\frac{d_e}{1+\frac{\on{ev}_{e}^*c_1(L)-d_e\bar{\psi}_{(e,v_0)}}{t}}\left(\sum_{i=0}^{\infty}(t/r)^{-i}c_i(-R^*\pi_*\mathcal L)\right);
\end{align}

\item if the target expands over the infinity section, there is a factor
\begin{align}
\frac{\prod_{e\in E(\Gamma)}d_e}{-t-\psi_\infty}=-\frac{1}{t}\frac{\prod_{e\in E(\Gamma)}d_e}{1+\frac{\psi_\infty}{t}}.
\end{align}
\end{itemize}

Now we want to take the coefficient of $t^0$. If we have any stable vertex over zero that does not contain mid-ages or stable vertex over infinity, then we will end up with negative powers of $t$. Therefore, the only stable vertex is the vertex $v_0$ containing mid-ages. We have
\begin{align}\label{fixed-loci-genus-zero}
\overline{M}_{\Gamma}=\overline{M}_{0,\vec k,k_a,k_b,\vec \mu,\pi_*\beta}(\mathcal D_0).
\end{align}
Recall that $\tau$ is the forgetful map that forgets the orbifold structures and we consider the pushforward of (\ref{localization-genus-zero}) by $\tau$. Therefore, we pushforward the virtual cycle  of (\ref{fixed-loci-genus-zero}) to the moduli space of stable maps to $D$. Note that $\mathcal D_0$ is a $\mu_r$-gerbe over $D$. By \cite{AJT15}*{Theorem 4.3}, the pushforward of the virtual cycle of (\ref{fixed-loci-genus-zero}) is simply
\[
\frac{1}{r}[\overline{M}_{0,m+l(\mu)+2,\pi_*\beta}(D)]^{\on{vir}}.
\]

Therefore, after multiplying by $r$, genus zero orbifold invariants with a pair of mid-ages stabilize for sufficiently large $r$ and the value does not depend on the choice of mid-ages. Combining with the degeneration formula, there is a positive integer $d_0$ 
such that, orbifold invariants of $X_{D,r}$ with a pair of mid-ages are constant in $r$ and do not depend on $k_a$ as long as $d_0<k_a\leq k_b$. This proves Theorem \ref{thm-genus-zero-mid-age} when $m_-=0$.
\end{proof}

Following the idea in \cite{FWY}*{Appendix A}, equivariant theory can be considered as a limit of non-equivariant theory. The following equivariant version of the result for $m_-=0$ can be proved by following \cite{FWY}*{Appendix A}.
\begin{lemma}\label{lemma-equivariant-genus-zero}
Genus zero equivariant class
\[
r\tau_*\left(\left[\overline{M}_{0,\vec k, k_a,k_b,\vec \mu,\beta}(Y_{D_0,r},D_\infty)\right]_{\mathbb C^*}^{\on{vir}}\right)
\]  
is constant in $r$ for sufficiently large $r$. Furthermore, there exists a positive integer $d_0$ such that this equivariant class does not depend on the pair $(k_a,k_b)$ as long as $d_0<k_a\leq k_b$.
\end{lemma}
\begin{proof}
By \cite{EG}*{Section 2.2}, the $i$-th Chow group of a space $M$ under an algebraic group $G$ can be defined as the following. Let $V$ be an $l$-dimensional representation of $G$ and $U\subset V$ be an equivariant open set where $G$ acts freely and whose complement has codimension more than $\dim M-i$. Then the $i$-th Chow group is defined as
\begin{align}\label{equiv-chow}
A_i^G(M)=A_{i+l-\dim G}((M\times U)/G).
\end{align}
For our purpose, we let $G=\mathbb C^*$ and $E=U=\mathbb C^N-\{0\}$, where $N$ is a sufficiently large integer. Then we have that $(M\times E)/\mathbb C^*$ is an $M$-fibration over $B=U/G=\mathbb P^{N-1}$. 

Note that the smooth divisor $D_\infty\subset Y$ induces a smooth divisor
\[
D_\infty\times_{\mathbb C^*} E \subset Y_{D_0,r}\times_{\mathbb C^*}E.
\]
There is the projection
\[
\pi: Y_{D_0,r}\times_{\mathbb C^*}E \rightarrow B.
\]
We consider the genus zero Gromov--Witten theory of $Y_{D_0,r}\times_{\mathbb C^*}E$ relative to $ D_\infty\times_{\mathbb C^*} E$ and choose the curve class $\beta$ such that $\pi_*\beta=0$. We have the isomorphism: 
\begin{align}\label{isom-genus-0-equiv}
\overline{M}_{0,\vec k,k_a,k_b,\vec \mu,\beta}(Y_{D_0,r}\times_{\mathbb C^*}E, D_\infty\times_{\mathbb C^*} E)\cong \left(\overline{M}_{0,\vec k,k_a,k_b,\vec \mu,\beta}(Y_{D_0,r},D_\infty)\times E\right)/\mathbb C^*.
\end{align}
In genus zero, there are natural perfect obstruction theories on both sides and they are identified under this isomorphism. By Theorem \ref{thm-genus-zero-mid-age} with $m_-=0$, the class
\[
r\tau_*\left[\overline{M}_{0,\vec k,k_a,k_b,\vec \mu,\beta}(Y_{D_0,r}\times_{\mathbb C^*}E, D_\infty\times_{\mathbb C^*} E)\right]^{\on{vir}}
\]
is constant in $r$ for sufficiently large $r$ and there exists a positive integer $d_0$ such that this class does not depend on the pair $(k_a,k_b)$ as long as $k_a>d_0$. Therefore, the same statement is true for the right-hand side of (\ref{isom-genus-0-equiv}). By (\ref{equiv-chow}) when $M=\overline{M}_{0,\vec k, k_a,k_b,\vec \mu,\beta}(Y_{D_0,r},D_\infty)$, for $N$ sufficiently large, the Chow group of the right-hand side of (\ref{isom-genus-0-equiv}) is isomorphic to equivariant Chow groups of the moduli space
$\overline{M}_{0,\vec k, k_a,k_b,\vec \mu,\beta}(Y_{D_0,r},D_\infty)$.
Under this identification, by the construction of the virtual class, the virtual classes of (\ref{isom-genus-0-equiv}) is identified with the equivariant virtual class of 
$\overline{M}_{0,\vec k, k_a,k_b,\vec \mu,\beta}(Y_{D_0,r},D_\infty)$.
This concludes the proof.
\end{proof}
\subsection{Invariants with large ages}

In this section, we prove Theorem \ref{thm-genus-zero-mid-age} when $m_->0$. The proof uses the result for $m_-=0$ in the previous section.

The following result for cycle class holds because of Lemma \ref{lemma-equivariant-genus-zero} and the proof of \cite{FWY}*{Lemma A.1}.
\begin{lemma}\label{lemma-cycle}
Given any partition $\vec k$ of $\int_\beta [D_0]$. For any positive integer $j$ and $r\gg 1$, the following class
\[
r^{j+1}(\tau)_*\left(c_j(-R^*\pi_*\mathcal L)\cap \left[\overline{M}_{0,\vec k, k_a,k_b,\beta}(\mathcal D_0)\right]^{vir}\right)
\]
is constant in $r$. Furthermore, there exists a positive integer $d_0$ such that this cycle class does not depend on the pair $(k_a,k_b)$ as long as $d_0<k_a\leq k_b$. 
\end{lemma}

\begin{proof}
We prove it by taking localization residues of $\overline{M}_{0,\vec k, k_a,k_b,\vec \mu,\beta}(Y_{D_0,r},D_\infty)$. We consider the decorated graph with one vertex over the zero divisor such that ages for markings and edges are given by $\vec k, k_a, k_b$ and $\vec u$. By Lemma \ref{lemma-equivariant-genus-zero}, the cycle 
\[
r\tau_*\left(\sum_{j=0}^{\infty}(t/r)^{-j}c_j(-R^*\pi_*\mathcal L)\right),
\]
coming from the localization residue computed in (\ref{vertex-contri-genus-zero-mid}) is constant in $r$ for sufficiently large $r$ and there exists a positive integer $d_0$ such that this cycle does not depend on the pair $(k_a,k_b)$ as long as $d_0<k_a\leq k_b$. Therefore, the lemma holds.
\end{proof}

\begin{proof}[Proof of Theorem \ref{thm-genus-zero-mid-age} with $m_-\neq 0$.]
Now we turn to the proof of Theorem \ref{thm-genus-zero-mid-age} for invariants with large ages. Again, we use the localization formula (\ref{localization-genus-zero}). Similar to the proof of Lemma \ref{lemma-d-0},
there exists a positive integer $d_0$ such that when $d_0<k_a\leq k_b$, the two orbifold markings $p_a$ and $p_b$ lie in the same contracted component over the zero divisor.

The localization contributions are

\begin{itemize}
\item for each stable vertex $v$ over the zero section that does not contain mid-ages, there is a factor
\begin{align}
&\prod_{e\in E(v)}\frac{rd_e}{t+\on{ev}_{e}^*c_1(L)-d_e\bar{\psi}_{(e,v)}}\left(\sum_{i=0}^{\infty}(t/r)^{-1+|E(v)|+m_-(v)-i}c_i(-R^*\pi_*\mathcal L)\right)\\
\notag =&\frac rt \prod_{e\in E(v)}\frac{d_e}{1+\frac{\on{ev}_{e}^*c_1(L)-d_e\bar{\psi}_{(e,v)}}{t}}\left(\sum_{i=0}^{\infty}(t/r)^{m_-(v)-i}c_i(-R^*\pi_*\mathcal L)\right);
\end{align}
\item for the stable vertex $v_0$ over the zero section containing mid-ages, there is a factor
\begin{align}
&\prod_{e\in E(v_0)}\frac{rd_e}{t+\on{ev}_{e}^*c_1(L)-d_e\bar{\psi}_{(e,v_0)}}\left(\sum_{i=0}^{\infty}(t/r)^{|E(v_0)|+m_-(v_0)-i}c_i(-R^*\pi_*\mathcal L)\right)\\
\notag =& \prod_{e\in E(v_0)}\frac{d_e}{1+\frac{\on{ev}_{e}^*c_1(L)-d_e\bar{\psi}_{(e,v_0)}}{t}}\left(\sum_{i=0}^{\infty}(t/r)^{m_-(v_0)-i}c_i(-R^*\pi_*\mathcal L)\right);
\end{align}

\item if the target expands over the infinity section, there is a factor
\begin{align}
\frac{\prod_{e\in E(\Gamma)}d_e}{-t-\psi_\infty}=-\frac{1}{t}\frac{\prod_{e\in E(\Gamma)}d_e}{1+\frac{\psi_\infty}{t}}.
\end{align}
\end{itemize}

Similar to the case with $m_-=0$, we take the coefficient of $t^0$. Then Lemma \ref{lemma-cycle} and \cite{FWY}*{Lemma A.1} together imply that the class
\[
r^{m_-+1}\tau_*\left(\left[\overline{M}_{0,\vec k, k_a,k_b,\beta}(X_{D,r})\right]^{\on{vir}}\right)
\] 
is constant in $r$ and independent of $k_a$ and $k_b$ when $d_0<k_a\leq k_b$. This completes the proof of Theorem \ref{thm-genus-zero-mid-age}.
\end{proof}

\section{Higher genus}\label{sec:higher-genus}

In this section, we consider higher genus orbifold invariants of $X_{D,r}$ with a pair of mid-ages. We will prove Theorem \ref{thm-higher-genus-mid-age} which states that genus $g$ orbifold invariants with a pair of mid-ages are polynomials in $k_a$ and the $k_a^i$-coefficients are polynomials in $r$ with degree bounded by $2g$ for sufficiently large $r$. In particular, this implies that genus zero invariants are constant in $r$.

\subsection{Invariants without large ages}

Similar to the genus zero case, we first consider invariants without large ages. In other words, $m_-=0$. 

\begin{proof}[Proof of Theorem \ref{thm-higher-genus-mid-age} with $m_-=0$.]
Let $d_0$ be the positive integer in Lemma \ref{lemma-d-0}. When $k_a\leq d_0$, by \cite{TY18}, \cite{FWY19} and \cite{TY20}, for sufficiently large $r$, the cycle
\[
r\cdot \tau_*\left(\left[\overline{M}_{g,\vec k, k_a,k_b,\beta}(X_{D,r})\right]^{\on{vir}}\right)
\]
is a polynomial in $r$ with degree bounded by $2g-1$ and the constant term is the corresponding relative Gromov--Witten cycle with one negative contact order.

We consider orbifold invariants with a pair $(k_a,k_b)$ of mid-ages such that $k_a>d_0$. By Lemma \ref{lemma-d-0}, two extra orbifold markings $p_a$ and $p_b$ are in the same contracted component. The localization formula is

\begin{align}
[\overline{M}_{g,\vec k,k_a,k_b,\vec \mu,\beta}(Y_{D_0,r}, D_\infty)]^{\on{vir}}= \qquad \sum_{\Gamma}\frac{1}{|\on{Aut}(\Gamma)|\prod_{e\in E(\Gamma)}d_e} \cdot\iota_*\left(\frac{[\overline{M}_{\Gamma}]^{\on{vir}}}{e(\on{Norm}_{\Gamma}^{\on{vir}})}\right).
\end{align}
The inverse of the virtual normal bundle $\frac{1}{e(\on{Norm}_{\Gamma}^{\on{vir}})}$ can be written as the product of the following factors:
\begin{itemize}
\item for each stable vertex $v$ over the zero section that does not contain mid-ages, there is a factor
\begin{align}
&\prod_{e\in E(v)}\frac{rd_e}{t+\on{ev}_{e}^*c_1(L)-d_e\bar{\psi}_{(e,v)}}\left(\sum_{i=0}^{\infty}(t/r)^{-1+g(v)+|E(v)|-i}c_i(-R^*\pi_*\mathcal L)\right)\\
\notag =&\frac rt \prod_{e\in E(v)}\frac{d_e}{1+\frac{\on{ev}_{e}^*c_1(L)-d_e\bar{\psi}_{(e,v)}}{t}}\left(\sum_{i=0}^{\infty}(t/r)^{-i+g(v)}c_i(-R^*\pi_*\mathcal L)\right)\\
\notag =&\frac 1t \prod_{e\in E(v)}\frac{d_e}{1+\frac{\on{ev}_{e}^*c_1(L)-d_e\bar{\psi}_{(e,v)}}{t}}\left(\sum_{i=0}^{\infty}(tr)^{-i+g(v)}r^{2i-2g(v)+1}c_i(-R^*\pi_*\mathcal L)\right),
\end{align}
\item For the stable vertex $v_0$ over the zero section containing mid-ages, there is a factor
\begin{align}
&\prod_{e\in E(v_0)}\frac{rd_e}{t+\on{ev}_{e}^*c_1(L)-d_e\bar{\psi}_{(e,v_0)}}\left(\sum_{i=0}^{\infty}(t/r)^{g(v_0)+|E(v_0)|-i}c_i(-R^*\pi_*\mathcal L)\right)\\
\notag =& \prod_{e\in E(v_0)}\frac{d_e}{1+\frac{\on{ev}_{e}^*c_1(L)-d_e\bar{\psi}_{(e,v_0)}}{t}}\left(\sum_{i=0}^{\infty}(t/r)^{-i+g(v_0)}c_i(-R^*\pi_*\mathcal L)\right)\\
\notag =& \frac{1}{r}\prod_{e\in E(v_0)}\frac{d_e}{1+\frac{\on{ev}_{e}^*c_1(L)-d_e\bar{\psi}_{(e,v_0)}}{t}}\left(\sum_{i=0}^{\infty}(tr)^{-i+g(v_0)}r^{2i-2g(v_0)+1}c_i(-R^*\pi_*\mathcal L)\right).
\end{align}

\item If the target expands over the infinity section, there is a factor
\begin{align}
\frac{\prod_{e\in E(\Gamma)}d_e}{-t-\psi_\infty}=-\frac{1}{t}\frac{\prod_{e\in E(\Gamma)}d_e}{1+\frac{\psi_\infty}{t}}.
\end{align}
\end{itemize}

By the same computation in \cite{JPPZ18}*{Section 2.3 and 2.4} and \cite{TY20}*{Lemma 2} using the Grothendieck–Riemann–Roch formula, the class
\[
r^{2i-2g(v)+1}(\tau)_*\left(c_i(-R^*\pi_*\mathcal L)\right)
\]
is a polynomial in $r$ with degree bounded by $2i$ for $v\neq v_0$. For $v=v_0$, the Grothendieck–Riemann–Roch formula again implies that the class is a polynomial in $k_a$, the $k_a^i$-coefficient of the polynomial in $k_a$ is a polynomial in $r$ with degree bounded by $2i$. We consider the pushforward of $r[\overline{M}_{g,\vec k,k_a,k_b,\vec \mu,\beta}(Y_{D_0,r}, D_\infty)]_{k_a^i}^{\on{vir}}$ to the moduli space of stable maps to $D$ and take the coefficient of $t^0$ of the total localization contribution. Note that the terms with negative power of $r$ also have negative power of $t$. So we have a polynomial in $r$. Therefore, following the proof of \cite{TY20}*{Proposition 2.1}, the pushforward of $r[\overline{M}_{g,\vec k,k_a,k_b,\vec \mu,\beta}(Y_{D_0,r}, D_\infty)]_{k_a^i}^{\on{vir}}$ is a polynomial in $r$ with degree bounded by $2g$. This completes the proof of Theorem \ref{thm-higher-genus-mid-age} when $m_-=0$.
\end{proof}

\begin{lemma}\label{lemma-higher-genus-equiv}
The equivariant class
\[
r\cdot\tau_*\left(\left[\overline{M}_{g,\vec k, k_a,k_b,\beta}(Y_{D_0,r},D_\infty)\right]_{\mathbb C^*}^{\on{vir}}\right)
\]  
is a polynomial in $k_a$ and the coefficients of the polynomial in $k_a$ are polynomials in $r$ with degree bounded by $2g$ for sufficiently large $r$. 
\end{lemma}

\begin{proof}
This is the higher genus version of Lemma \ref{lemma-equivariant-genus-zero}. The proof is parallel to the proof of Lemma \ref{lemma-equivariant-genus-zero} with minor modifications in higher genus case as described in details in \cite{FWY19}*{Section 4}. 
\end{proof}

\subsection{Invariants with large ages}
For invariants with large ages, the proof is similar to the proof in the previous section with minor modifications for the higher genus case mentioned in \cite{FWY19}*{Section 4}.

The following result for cycle class holds by using the proof of \cite{FWY19}*{Corollary 4.2}, \cite{TY20}*{Lemma 2} and \cite{TY20}*{Lemma 3}.
\begin{lemma}\label{lemma-cycle-higher}
Given any partition $\vec k$ of $\int_\beta [D_0]$. For any positive integer $j$ and $r\gg 1$, the following class
\[
r^{j+1-g}(\tau)_*\left(c_j(-R^*\pi_*\mathcal L)\cap \left[\overline{M}_{g,\vec k, k_a,k_b,\beta}(\mathcal D_0)\right]^{\on{vir}}\right)
\]
is a  polynomial in $k_a$ and the coefficients of the polynomial in $k_a$ are polynomials in $r$ with degree bounded by $\min\{2j,2g\}$.
\end{lemma}
\begin{proof}
When $j\leq g$, the proof follows from that of \cite{TY20}*{Lemma 2}, so the degree bound is $2j$. When $j >g$, the proof is similar to that of Lemma \ref{lemma-cycle}. That is, using Lemma \ref{lemma-higher-genus-equiv} and taking the localization residue, we have that the class is a polynomial in $r$ with degree bounded by $2g$. This completes the proof.
\end{proof}

Then the proof of Theorem \ref{thm-higher-genus-mid-age} follows from localization computation as in previous sections. 

\begin{proof}[Proof of Theorem \ref{thm-higher-genus-mid-age} with $m_-\neq 0$.]
The localization contributions are
\begin{itemize}
\item for each stable vertex $v$ over the zero section that does not contain mid-ages, there is a factor
\begin{align}
&\prod_{e\in E(v)}\frac{rd_e}{t+\on{ev}_{e}^*c_1(L)-d_e\bar{\psi}_{(e,v)}}\left(\sum_{i=0}^{\infty}(t/r)^{-1+g(v)+|E(v)|+m_-(v)-i}c_i(-R^*\pi_*\mathcal L)\right)\\
\notag =&\frac rt \prod_{e\in E(v)}\frac{d_e}{1+\frac{\on{ev}_{e}^*c_1(L)-d_e\bar{\psi}_{(e,v)}}{t}}\left(\sum_{i=0}^{\infty}(t/r)^{-i+g(v)+m_-(v)}c_i(-R^*\pi_*\mathcal L)\right)\\
\notag =&\frac {t^{m_-(v)-1}}{r^{m_-(v)}} \prod_{e\in E(v)}\frac{d_e}{1+\frac{\on{ev}_{e}^*c_1(L)-d_e\bar{\psi}_{(e,v)}}{t}}\left(\sum_{i=0}^{\infty}t^{-i+g(v)}r^{i-g(v)+1}c_i(-R^*\pi_*\mathcal L)\right);
\end{align}
\item for the stable vertex $v_0$ over the zero section containing mid-ages, there is a factor
\begin{align}
&\prod_{e\in E(v_0)}\frac{rd_e}{t+\on{ev}_{e}^*c_1(L)-d_e\bar{\psi}_{(e,v_0)}}\left(\sum_{i=0}^{\infty}(t/r)^{g(v_0)+|E(v_0)|+m_-(v_0)-i}c_i(-R^*\pi_*\mathcal L)\right)\\
\notag =& \prod_{e\in E(v_0)}\frac{d_e}{1+\frac{\on{ev}_{e}^*c_1(L)-d_e\bar{\psi}_{(e,v_0)}}{t}}\left(\sum_{i=0}^{\infty}(t/r)^{-i+g(v_0)+m_-(v_0)}c_i(-R^*\pi_*\mathcal L)\right)\\
\notag =& \frac{t^{m_-(v_0)}}{r^{m_-(v_0)+1}}\prod_{e\in E(v_0)}\frac{d_e}{1+\frac{\on{ev}_{e}^*c_1(L)-d_e\bar{\psi}_{(e,v_0)}}{t}}\left(\sum_{i=0}^{\infty}t^{-i+g(v_0)}r^{i-g(v_0)+1}c_i(-R^*\pi_*\mathcal L)\right);
\end{align}

\item if the target expands over the infinity section, there is a factor
\begin{align}
\frac{\prod_{e\in E(\Gamma)}d_e}{-t-\psi_\infty}=-\frac{1}{t}\frac{\prod_{e\in E(\Gamma)}d_e}{1+\frac{\psi_\infty}{t}}.
\end{align}
\end{itemize}

We want to take the coefficient of $t^0$ of the total localization contribution. Again, the class
\[
r^{2i-2g(v)+1}(\tau)_*\left(c_i(-R^*\pi_*\mathcal L)\right)
\]
is a polynomial in $r$ with degree bounded by $2i$. Together with the polynomiality of Lemma \ref{lemma-cycle-higher}, we obtain that the pushforward of 
\[
r^{m_-+1}[\overline{M}_{g,\vec k,k_a,k_b,\vec \mu,\beta}(Y_{D_0,r}, D_\infty)]^{\on{vir}}
\]
to the moduli space of stable maps to $D$ is a  polynomial in $k_a$ and the coefficients of the polynomial in $k_a$ are polynomials in $r$ with degree bounded by $2g$. This completes the proof of Theorem \ref{thm-higher-genus-mid-age}.
\end{proof}

\if{
\begin{remark}
By \cite{FWY19} and \cite{TY20}, for sufficiently large $r$, invariants with $k_a\leq d_0$ will be polynomials in $r$ with degree bounded by $2g-1$ and the constant terms are relative invariants with negative contact orders. As $k_a$ gets bigger, the proof of \cite{FWY19} and \cite{TY20} will not work eventually, because we will need bigger $r$. They are polynomials in $r$ with degree bounded by $2g$ instead of $2g-1$. 
It gives another explanation for the degree bound of orbifold invariants of $X_{D,r}$ in \cite{TY20}. In \cite{TY20}, it is proved that genus $(g+1)$ orbifold invariants of root stacks (without mid-ages) are polynomials of degree bounded by $2g+1$. By the loop axiom for orbifold Gromov--Witten theory, they can be written as sum of genus $g$ orbifold invariants. The summation includes genus $g$ orbifold invariants with a pair of mid-ages which are polynomials whose degrees are bounded by $2g$. And extra factor of $r$ comes in after summing over invariants with mid-ages.
\end{remark}
}\fi

\section{The loop axiom for relative Gromov--Witten theory}

\subsection{CohFT}\label{sec-cohft}
Let $\overline {M}_{g,m}$ be the moduli space of genus $g$, $m$-pointed stable curves. We assume that $2g-2+m>0$. There are several canonical morphisms between moduli space $\overline{M}_{g,m}$ of stable curves. 
\begin{itemize}
    \item There is a forgetful morphism
    \[
    \pi:\overline{M}_{g,m+1}\rightarrow \overline{M}_{g,m}
    \]
    obtained by forgetting the last marking of $(m+1)$-pointed, genus $g$ curves in $\overline{M}_{g,m+1}$.
    \item There is a morphism of gluing the loop
    \[
    \rho_l: \overline{M}_{g,m+2}\rightarrow \overline{M}_{g+1,m}
    \]
    obtained by identifying the last two markings of the $(m+2)$-pointed, genus $g$ curves in $ \overline{M}_{g,m+2}$.
    \item There is a morphism of gluing the tree
    \[
    \rho_t:\overline{M}_{g_1,m_1+1}\times \overline{M}_{g_2,m_2+1}\rightarrow \overline{M}_{g_1+g_2,m_1+m_2}
    \]
    obtained by identifying the last markings of separate pointed curves in $\overline{M}_{g_1,m_1+1}\times \overline{M}_{g_2,m_2+1}$.
\end{itemize}

The state space $H$ is a graded vector space with a non-degenerate pairing $\langle ,\rangle$ and a distinguished element $1\in H$. Given a basis $\{e_i\}$, let $\eta_{jk}=\langle e_j,e_k\rangle$ and $(\eta^{jk})=(\eta_{jk})^{-1}$.

A cohomological field theory (CohFT) is a collection of homomorphisms
\[
\Omega_{g,m}: H^{\otimes m}\rightarrow H^*(\overline{M}_{g,m},\mathbb Q)
\]
satisfying the following axioms:
\begin{itemize}
    \item The element $\Omega_{g,m}$ is invariant under the natural action of symmetric group $S_m$.
    \item For all $a_i\in H$, $\Omega_{g,m}$ satisfies
    \[
    \Omega_{g,m+1}(a_1,\ldots,a_m,1)=\pi^*\Omega_{g,m}(a_1,\ldots,a_m).
    \]
    \item The splitting axiom:
    \begin{align*}
    &\rho^*_t\Omega_{g_1+g_2,m_1+m_2}(a_1,\ldots,a_{m_1+m_2})=\\
    &\sum_{j,k}\eta^{jk}\Omega_{g_1,m_1}(a_1,\ldots,a_{m_1},e_j)\otimes \Omega_{g_2,m_2}(a_{m_1+1},\ldots,a_{m_1+m_2},e_k),
    \end{align*}
    for all $a_i\in H$.
    \item The loop axiom:
    \[
    \rho_l^*\Omega_{g+1,m}(a_1,\ldots,a_m)=\sum_{j,k}\eta^{jk}\Omega_{g,m+2}(a_1,\ldots,a_m,e_j,e_k),
    \]
    for all $a_i\in H$. In addition, the equality
    \[
    \Omega_{0,3}(v_1,v_2,1)=\langle v_1,v_2\rangle
    \]
    holds for all $v_1,v_2\in H$.
\end{itemize}

A natural example for CohFT is Gromov--Witten theory. In particular, orbifold Gromov--Witten theory of the $r$-th root stack $X_{D,r}$ is a CohFT and the state space is the Chen--Ruan cohomology  ring $H^*(IX_{D,r})$. Note that we ignore odd cohomology classes. A CohFT can be formulated with signs in the presence of odd cohomology classes.

Following \cite{LRZ}, a CohFT without the loop axiom is called a partial CohFT. Following \cite{FWY19}*{Theorem 3.16}, relative Gromov--Witten theory defined in \cite{FWY19}*{Section 3} is a partial CohFT. The construction is as follows.

The ring of insertions (state space) for relative Gromov--Witten theory, defined in \cite{FWY}, is
\[
\HH=\bigoplus\limits_{i\in\mathbb Z}\HH_i,
\]
where $\HH_0=H^*(X)$ and $\HH_i=H^*(D)$ if $i\in \mathbb Z - \{0\}$.
For an element $\gamma\in \HH_i$, we write $[\gamma]_i$ for its embedding in $\HH$. 

The pairing on $\HH$ is defined as follows:
\begin{equation}\label{eqn:pairing}
\begin{split}
([\gamma]_i,[\delta]_j) = 
\begin{cases}
0, &\text{if } i+j\neq 0;\\
\int_X \gamma\cup\delta, &\text{if } i=j=0; \\
\int_D \gamma\cup\delta, &\text{if } i+j=0, i,j\neq 0.
\end{cases}
\end{split}
\end{equation}

For a sufficiently large integer $r$, an element $[\gamma]_i$ naturally corresponds to a cohomology class $\gamma\in H^*(\underline{I}X_{D,r})$ that lies in either a component with age $i/r$ (if $i\geq0$), or a component with age $(r+i)/r$ (if $i<0$). 
We will write $H^*(\underline{I}X_{D,r})_i$ for the cohomology of the twisted sector with age $i/r$. 
To simplify the notation in the next definition, we later refer to a general element of $\HH$ by simply writing $[\gamma]\in \HH$ without a subscript. 

Consider the forgetful map 
\[
\pi:\overline{M}_{g,m}(X,\beta) \times_{X^{m_++m_-}} D^{m_++m_-} \rightarrow \overline{M}_{g,m},
\]
where a list of contact orders will be implied in the context, and the fiber product remembers the $(m_++m_-)$ markings that correspond to relative markings.

\begin{defn}\label{defn-relative-GW-class}
Given elements $[\gamma_1],\ldots,[\gamma_m]\in \HH$, the relative Gromov--Witten class is defined as
\[
\Omega^{(X,D)}_{g,m,\beta}([\gamma_1],\ldots,[\gamma_m])=\pi_*\left(\prod_{i=1}^m\overline{\on{ev}}_i^*(\gamma_i)\cap\mathfrak c_\Gamma(X/D)\right)\in H^*(\overline{M}_{g,m},\mathbb Q),
\]
where the topological type $\Gamma$ is determined by $g,m,\beta$ and the insertions $[\gamma_1],\ldots,[\gamma_m]\in \HH$. We refer to \cite{FWY19}*{Section 3.3} for the definition of the evaluation map $\overline{\on{ev}}_ i$ and the class 
\[
\mathfrak c_\Gamma(X/D)\in \overline{M}_{g,m}(X,\beta) \times_{X^{m_++m_-}} D^{m_++m_-}.
\] 
We then define the class
\[
\Omega^{(X,D)}_{g,m}([\gamma_1],\ldots,[\gamma_m])=\sum_{\beta\in H_2(X,\mathbb Z)}\Omega^{(X,D)}_{g,m,\beta}([\gamma_1],\ldots,[\gamma_m])q^\beta.
\]
\end{defn}

\subsection{Genus one}

\begin{proof}[Proof of Theorem \ref{thm-genus-one-loop-axiom}]
The loop axiom for genus one orbifold Gromov--Witten theory of root stack $X_{D,r}$ is 
\begin{align*}
\rho_l^*\Omega^{X_{D,r}}_{1,m}(\gamma_1,\ldots,\gamma_m)=&\sum_{j,k}\eta^{jk}\Omega^{X_{D,r}}_{0,m+2}(\gamma_1,\ldots,\gamma_m,e_j,e_k).
\end{align*}

Let $p_j$ and $p_k$ be the two new markings corresponding to $e_j$ and $e_k$. Let $k_a/r$ and $k_b/r$ be the corresponding ages for $p_j$ and $p_k$ respectively. Then we need to have $k_a+k_b=r$, otherwise we have $\eta^{jk}=0$. Now, we do not assume that $k_a\leq k_b$.

By Theorem \ref{thm-genus-zero-mid-age},  there are sufficiently large $r$ such that, for $g=0$ and $d_0<k_a<r-d_0$, the orbifold Gromov--Witten classes are the same and independent of $r$. Moreover, when $k_a\leq d_0$ or $k_a\geq r-d_0$, genus zero orbifold Gromov--Witten classes coincide with genus zero relative Gromov--Witten classes with negative contact orders \cite{FWY}. Note that for the orbifold Gromov--Witten theory, the orbifold pairing for the $\mu_r$-gerbe $D_r$ requires an extra factor of $r$. This is consistent with the result in \cite{FWY} which states that, in order to relate relative and orbifold invariants when there is a large age marking, we need to multiple orbifold invariants by a factor of $r$. 
Similarly, in Theorem \ref{thm-genus-zero-mid-age}, orbifold invariants are also multiplied by an extra factor of $r$ when we add two mid-age markings.

Therefore, the loop axiom for genus one orbifold Gromov--Witten theory implies the loop axiom for genus one relative Gromov--Witten class as follows.

\begin{align*}
&r^{m_-}\rho_l^*\Omega^{X_{D,r}}_{1,m}(\gamma_1,\ldots,\gamma_m)\\
=&\sum_{j,k: e_j\in H^*(\underline{I}X_{D,r})_{k_a}, k_a\leq d_0, \text{ or } k_a\geq r-d_0}\eta^{jk}r^{m_-}\Omega^{X_{D,r}}_{0,m+2}(\gamma_1,\ldots,\gamma_m,e_j,e_k)\\
&+\sum_{j,k: e_j\in H^*(\underline{I}X_{D,r})_{k_a},d_0< k_a<r- d_0}\eta^{jk}r^{m_-}\Omega^{X_{D,r}}_{0,m+2}(\gamma_1,\ldots,\gamma_m,e_j,e_k)\\
=& \sum_{j,k: [e_j]\in \HH_{k_a}, k_a\leq d_0, \text{ or }k_a\geq r-d_0}\eta^{jk}\Omega^{(X,D)}_{0,m+2}([\gamma_1],\ldots,[\gamma_m],[e_j],[e_k])\\
&+(r-2d_0-1)\sum_{j,k: e_j\in H^*(\underline{I}X_{D,r})_{d_0+1}}\eta^{jk}r^{m_-}\Omega^{X_{D,r}}_{0,m+2}(\gamma_1,\ldots,\gamma_m,e_j,e_k).
\end{align*}

Therefore, we have the following system of equations
\begin{align*}
&\left[r^{m_-}\rho_l^*\Omega^{X_{D,r}}_{1,m}(\gamma_1,\ldots,\gamma_m)\right]_{r^0}\\
=&\sum_{j,k: [e_j]\in \HH_{k_a}, k_a\leq d_0 \text{ or }k_a\geq r-d_0}\eta^{jk}\Omega^{(X,D)}_{0,m+2}([\gamma_1],\ldots,[\gamma_m],[e_j],[e_k])\\
&-(2d_0+1)\sum_{j,k: e_j\in H^*(\underline{I}X_{D,r})_{ d_0+1}}\eta^{jk}r^{m_-}\Omega^{X_{D,r}}_{0,m+2}(\gamma_1,\ldots,\gamma_m,e_j,e_k).
\end{align*}
and 
\[
\left[r^{m_-}\rho_l^*\Omega^{X_{D,r}}_{1,m}(\gamma_1,\ldots,\gamma_m)\right]_{r^1}=\sum_{j,k: e_j\in H^*(\underline{I}X_{D,r})_{d_0+1}}\eta^{jk}r^{m_-}\Omega^{X_{D,r}}_{0,m+2}(\gamma_1,\ldots,\gamma_m,e_j,e_k).
\]
Hence,
\begin{align*}
&\rho_l^*\Omega^{(X,D)}_{1,m}([\gamma_1],\ldots,[\gamma_m])\\
=&\sum_{j,k: [e_j]\in \HH_{k_a}, k_a\leq d_0 \text{ or } k_a\geq r-d_0}\eta^{jk}\Omega^{(X,D)}_{0,m+2}([\gamma_1],\ldots,[\gamma_m],[e_j],[e_k])\\
&-(2d_0+1)\left[r^{m_-}\rho_l^*\Omega^{X_{D,r}}_{1,m}(\gamma_1,\ldots,\gamma_m)\right]_{r^1}.
\end{align*}

For any $d^\prime$ that satisfies $d_0<d^\prime<r/2$, we can also write the loop axiom as follows
\begin{align*}
&\rho_l^*\Omega^{(X,D)}_{1,m}([\gamma_1],\ldots,[\gamma_m])\\
=&\sum_{j,k: [e_j]\in \HH_{k_a}, k_a\leq d^\prime \text{ or }k_a\geq r-d^\prime}\eta^{jk}\Omega^{(X,D)}_{0,m+2}([\gamma_1],\ldots,[\gamma_m],[e_j],[e_k])\\
&-(2d^\prime+1)\left[r^{m_-}\rho_l^*\Omega^{X_{D,r}}_{1,m}(\gamma_1,\ldots,\gamma_m)\right]_{r^1}.
\end{align*}
\end{proof}

\begin{remark}
The identity can be viewed as a modification of the usual loop axiom with a correction term which is given by the $r$-coefficient of the genus one orbifold Gromov--Witten class. The meanings of the non-constant coefficient of higher genus orbifold invariants are studied in \cite{TY20}. In particular, the $r$-coefficient of the genus one orbifold Gromov--Witten invariants can be expressed in terms of sum of genus zero relative Gromov--Witten invariants multiplied by genus one absolute Gromov--Witten invariants of the divisor $D$. Our result provides another explanation for the meaning of the $r$-coefficient of the genus one orbifold Gromov--Witten class.
\end{remark}

\begin{example}\label{example}
In \cite{FWY19}*{Example 3.17}, there is a counterexample for the loop axiom of relative Gromov--Witten theory of $(X,D)$, where $X=\mathbb P^1$ and $D$ is a point in X. Let $1$ be the identity class in $H^*(X)$ and $\omega$ be the point class in $X$. We have
\[
\Omega^{(X,D)}_{1,1,0}([1]_0)=1;
\]
\[
\Omega^{(X,D)}_{0,3,0}([1]_0,[1]_0,[\omega]_0)=1;
\]
\[
\Omega^{(X,D)}_{0,3,0}([1]_0,[1]_i,[1]_{-i})=1, i\in \mathbb{Z}^*.
\]
The regular loop axiom does not hold for the relative Gromov--Witten theory of $(X,D)$, but it holds for the corresponding orbifold Gromov--Witten theory. For orbifold Gromov--Witten theory, the loop axiom is
\begin{align*}
\Omega^{X_{D,r}}_{1,1,0}(1)&=2\Omega^{X_{D,r}}_{0,3,0}(1,1,\omega)+\sum_{i=1}^{r-1}\Omega^{X_{D,r}}_{0,3,0}(1,1_{i/r},r1_{(r-i)/r})\\
&=2+(r-1)\cdot 1
\end{align*}
Taking the constant term, we have:
\[
1=2+\left[(r-1)\cdot 1\right]_{r^0}.
\]
In other words, we have
\[
\Omega^{(X,D)}_{1,1,0}([1]_0)=2\Omega^{(X,D)}_{0,3,0}([1]_0,[1]_0,[\omega]_0)-1\cdot \left[\Omega^{X_{D,r}}_{1,1,0}(1)\right]_{r^1}.
\]
\end{example}

\subsection{Higher genus}\label{sec:higher-genus-loop}

\begin{proof}[Proof of Theorem \ref{thm-higher-genus-loop-axiom}]
Given a partition $\vec k$, we consider the loop axiom of genus $(g+1)$ orbifold Gromov--Witten theory of $X_{D,r}$ whose contact orders are given by $\vec k$. The loop axiom for genus $(g+1)$ orbifold Gromov--Witten theory of root stack $X_{D,r}$ is 
\begin{align}\label{loop-orb-g}
\rho_l^*\Omega^{X_{D,r}}_{g+1,m}(\gamma_1,\ldots,\gamma_m)=&\sum_{j,k}\eta^{jk}\Omega^{X_{D,r}}_{g,m+2}(\gamma_1,\ldots,\gamma_m,e_j,e_k).
\end{align}

For sufficiently large $r$, by \cite{FWY19} and \cite{TY20}, we know that 
\[
r^{m_-}\Omega^{X_{D,r}}_{g,m+2}(\gamma_1,\ldots,\gamma_m,e_j,e_k)
\] 
are polynomials in $r$ with degree bounded by $2g-1$ and the constant terms are relative Gromov--Witten classes with negative orders for $k_a\leq d_0$ or $k_a\geq r-d_0$. The sum of the rest of the terms on the right-hand side of (\ref{loop-orb-g}) is a polynomial in $r$ with degree bounded by $2g+1$. The correction term of the loop axiom of relative Gromov--Witten theory is given by
\begin{align}\label{higher-genus-correction}
C_{d_0}:=\left[\sum_{j,k: e_j\in H^*(\underline{I}X_{D,r})_{k_a},d_0<k_a<r-d_0} \eta^{jk}r^{m_-}\Omega^{X_{D,r}}_{g,m+2}(\gamma_1,\ldots,\gamma_m,e_j,e_k)\right]_{r^0}
\end{align}
The modified loop axiom can be written as
\begin{align*}
&\rho_l^*\Omega^{(X,D)}_{g+1,m}([\gamma_1],\ldots,[\gamma_m])\\
=&\sum_{j,k: [e_j]\in \HH_{k_a}, k_a\leq d_0 \text{ or }k_a\geq r-d_0}\eta^{jk}\Omega^{(X,D)}_{g,m+2}([\gamma_1],\ldots,[\gamma_m],[e_j],[e_k])+C_{d_0}.
\end{align*}
\end{proof}


We explain how to write the correction term $C_{d_0}$ more explicitly. For $e_j\in H^*(\underline{I}X_{D,r})_{k_a},k_a>d_0$, we can write
\[
\Omega^{X_{D,r}}_{g,m+2}(\gamma_1,\ldots,\gamma_m,e_j,e_k)=\sum_{i\geq 0}\left[\Omega^{X_{D,r}}_{g,m+2}(\gamma_1,\ldots,\gamma_m,e_j,e_k)\right]_{k_a^i}k_a^i.
\]
Then summing over $k_a$ for $d_0<k_a<r-d_0$, we have
\begin{align}\label{poly-C-0}
&\sum_{j,k: e_j\in H^*(\underline{I}X_{D,r})_{k_a},d_0<k_a<r-d_0} \eta^{jk}r^{m_-}\Omega^{X_{D,r}}_{g,m+2}(\gamma_1,\ldots,\gamma_m,e_j,e_k)\\
\notag=&\sum_{j,k: e_j\in H^*(\underline{I}X_{D,r})_{k_a},d_0<k_a<r-d_0} \sum_{i\geq 0}\eta^{jk}r^{m_-}\left[\Omega^{X_{D,r}}_{g,m+2}(\gamma_1,\ldots,\gamma_m,e_j,e_k)\right]_{k_a^i}k_a^i\\
\notag=&\sum_{j,k: e_j\in H^*(\underline{I}X_{D,r})_{k_a},k_a=d_0+1} \sum_{i\geq 0}\eta^{jk}r^{m_-}\left[\Omega^{X_{D,r}}_{g,m+2}(\gamma_1,\ldots,\gamma_m,e_j,e_k)\right]_{k_a^i}\left(\sum_{l=d_0+1}^{r-d_0-1}l^i\right)
\end{align}
Recall that, by Faulhaber's formula, we have
\begin{align}\label{Faulhaber}
\sum_{l=1}^r l^i=\frac{r^{i+1}}{i+1}+\frac{1}{2} r^i+\sum_{l=2}^i\frac{B_l}{l!}i^{\underline{l-1}}r^{i-l+1},
\end{align}
for $i>0$, where $i^{\underline{l-1}}=\frac{i!}{(i-l+1)!}$. Therefore, (\ref{Faulhaber}) is considered as a polynomial in $r$. Note that there is no constant term in $r$ in Faulhaber's formula. Therefore, for $i>0$, we have
\begin{align*}
\left[\left(\sum_{l=d_0+1}^{r-d_0-1}l^i\right)\right]_{r^0}&=-\left(\sum_{l=1}^{d_0}l^i\right)-\left[\left(\sum_{l=r-d_0}^{r}l^i\right)\right]_{r^0}\\
&=-\left((1+(-1)^i)\sum_{l=1}^{d_0}l^i\right).
\end{align*}
For $i=0$, we simply have
\[
\left[\left(\sum_{l=d_0+1}^{r-d_0-1}l^i\right)\right]_{r^0}=\left[r-2d_0-1\right]_{r^0}=-2d_0-1.
\]

Since $C_{d_0}$ is the constant term of (\ref{poly-C-0}), $C_{d_0}$ can be written as
\begin{align*}
&\left(-2d_0-1\right)\sum_{j,k: e_j\in H^*(\underline{I}X_{D,r})_{k_a},k_a=d_0+1} \eta^{jk}\left[r^{m_-}\left[\Omega^{X_{D,r}}_{g,m+2}(\gamma_1,\ldots,\gamma_m,e_j,e_k)\right]_{k_a^0}\right]_{r^0}+\\
&\sum_{i> 0}\left(-(1+(-1)^i)\sum_{l=1}^{d_0}l^i\right)\sum_{\substack{j,k: e_j\in H^*(\underline{I}X_{D,r})_{k_a}\\k_a=d_0+1}} \eta^{jk}\left[r^{m_-}\left[\Omega^{X_{D,r}}_{g,m+2}(\gamma_1,\ldots,\gamma_m,e_j,e_k)\right]_{k_a^i}\right]_{r^0}\\
&=\left(-2d_0-1\right)\sum_{j,k: e_j\in H^*(\underline{I}X_{D,r})_{k_a},k_a=d_0+1} \eta^{jk}\left[r^{m_-}\left[\Omega^{X_{D,r}}_{g,m+2}(\gamma_1,\ldots,\gamma_m,e_j,e_k)\right]_{k_a^0}\right]_{r^0}+\\
&\sum_{i> 0}\left(-2\sum_{l=1}^{d_0}l^{2i}\right)\sum_{\substack{j,k: e_j\in H^*(\underline{I}X_{D,r})_{k_a}\\k_a=d_0+1}} \eta^{jk}\left[r^{m_-}\left[\Omega^{X_{D,r}}_{g,m+2}(\gamma_1,\ldots,\gamma_m,e_j,e_k)\right]_{k_a^{2i}}\right]_{r^0}.
\end{align*}

It would be interesting to write the correction term in the loop axiom for higher genus relative Gromov--Witten theory in terms of a graph sum over moduli space of relative stable maps and rubber maps. Based on the result in genus one, we may also expect that the correction term in higher genus case is related to the coefficients of higher power of $r$. 

Since Virasoro constraints are at the level of invariants, it may be enough to consider the loop axiom at the level of invariants if one is only interested in its relation with Virasoro constraints. Here, we write down an example at the level of invariants where the correction terms are just zero.
\begin{example}\label{example-stationary}
We consider stationary invariants of curves.  Let $X=C$ be a smooth projective curve and $q$ be a point in $C$. We consider the root stack $C[r]$ of $C$ by taking $r$-th root along $q$. Here, we will use a slightly different notation for invariants. Let
\[
\vec k=(k_1,\ldots,k_m)\in (\mathbb Z_{>0})^m
\]
represent contact orders of relative marked points, where $m$ is the number of relative marked points. We also assume there are $n$ interior marked points with point constraints.
The stationary orbifold invariant of $C[r]$ is
\begin{align}\label{rel-inv-curve}
\left\langle \prod_{i=1}^n\tau_{a_i}(\omega)\right\rangle_{g,n,\vec k,d}^{C[r]},
\end{align}
where $\omega \in H^2(C,\mathbb Q)$ denote the class that is Poincar\'e dual to a point and $d\in H_2(C,\mathbb Z)$ is the degree of the curve class. By \cite{TY18}*{Theorem 1.9}, when $r$ is sufficiently large, stationary orbifold invariant (\ref{rel-inv-curve}) is constant in $r$ and equals to the corresponding relative invariant. By the loop axiom, we have
\[
\left\langle \prod_{i=1}^n\tau_{a_i}(\omega)\right\rangle_{g+1,n,\vec k,d}^{C[r]}=2\left\langle \prod_{i=1}^n\tau_{a_i}(\omega),1,\omega\right\rangle_{g,n+2,\vec k,d}^{C[r]}+\sum_{i=1}^{r-1}r\left\langle \prod_{i=1}^n\tau_{a_i}(\omega)\right\rangle_{g,n,\vec k,k_a,k_b,d}^{C[r]}.
\]

By the string equation, $\left\langle \prod_{i=1}^n\tau_{a_i}(\omega),1,\omega\right\rangle_{g,n+2,\vec k,d}^{C[r]}$ is also a stationary invariant, hence constant in $r$. 

For $\left\langle \prod_{i=1}^n\tau_{a_i}(\omega)\right\rangle_{g,n,\vec k,k_a,k_b,d}^{C[r]}$ with $k_a<k_j$ for some $k_j$ in $\vec k$, we degenerate $C$ into $C\cup_p \mathbb P^1$ such that the orbifold point $q$ is distributed to $\mathbb P^1$ and all stationary marked points are distributed to $C$. The degeneration formula \cite{Li2} can be used to write invariants of $C[r]$ in terms of sum of products of invariants of $(C,p)$ and (disconnected) invariants of $(\mathbb P^1[r],\infty)$. Note that there are no insertions for the invariant of $(\mathbb P^1[r],\infty)$, therefore the virtual dimension is zero. Similar to \cite{TY18}*{Section 5.1}, the invariants of $(\mathbb P^1[r],\infty)$ must be genus zero with one relative marked point, one orbifold marked point with small age, possibly one large age marked point, and no interior marked points. Note that there is only one large age marked point and the large age marked point does not affect the virtual dimension. Therefore, these genus zero invariants of $(\mathbb P^1[r],\infty)$ are simply genus zero invariants of $(\mathbb P^1,0,\infty)$ by \cite{FWY}, hence constant in $r$. In summary, we have
\[
r\left\langle \prod_{i=1}^n\tau_{a_i}(\omega)\right\rangle_{g,n,\vec k,k_a,k_b,d}^{C[r]}=\left\langle \prod_{i=1}^n\tau_{a_i}(\omega)\right\rangle_{g,n,\vec k,k_a,k_b,d}^{(C,q)}
\]
if $k_a<k_j$ for some $k_j$ in $\vec k$.

For $\left\langle \prod_{i=1}^n\tau_{a_i}(\omega)\right\rangle_{g,n,\vec k,k_a,k_b,d}^{C[r]}$ with $k_j<k_a\leq r/2$ for all $k_j$ in $\vec k$. The computation is similar to the previous case. The degeneration formula and the virtual dimension constraint again imply the invariants of $(\mathbb P^1[r],\infty)$ must be genus zero. For the component containing orbifold marked point with age $k_b/r$, it must also contain at least two more orbifold markings in order to satisfy the orbifold condition. Then the virtual dimension is at least 1 which is a contradiction. Therefore, 
\[
 \left\langle \prod_{i=1}^n\tau_{a_i}(\omega)\right\rangle_{g,n,\vec k,k_a,k_b,d}^{C[r]}=0,
\]
if $k_a>k_j$ for all $k_j$ in $\vec k$. 

As a result, the identity for the loop axiom for relative stationary invariants of curves can be written as
\[
\left\langle \prod_{i=1}^n\tau_{a_i}(\omega)\right\rangle_{g+1,n,\vec k,d}^{(C,q)}=2\left\langle \prod_{i=1}^n\tau_{a_i}(\omega),1,\omega\right\rangle_{g,n+2,\vec k,d}^{(C,q)}+2\sum_{k_a=1}^{d_0}\left\langle \prod_{i=1}^n\tau_{a_i}(\omega)\right\rangle_{g,n,\vec k,k_a,k_b,d}^{(C,q)},
\]
for $d_0>k_j$ for all $k_j$ in $\vec k$. In this case, the correction term is simply zero and mid-age invariants vanish. Therefore, we simply get a finite sum without any correction. This also provides another explanation why stationary orbifold invariants of target curves are constant in $r$.

\end{example}



\bibliographystyle{amsxport}
\bibliography{main}

\end{document}